\documentclass[12pt]{article}    
\usepackage{amsmath,amssymb,amsthm,amsfonts}

\usepackage{latexsym}

\usepackage{mathrsfs}
\usepackage{enumitem}
\usepackage{caption}
\usepackage{graphicx}
\usepackage{flafter}
\usepackage{geometry}
\usepackage{tikz}
\usepackage{ytableau}
\allowdisplaybreaks
\usepackage{verbatim}
\usepackage[colorlinks,linkcolor=blue,anchorcolor=blue,citecolor=blue]{hyperref}
\usepackage{color}

\numberwithin{equation}{section}
\numberwithin{figure}{section}

\hypersetup{colorlinks = true}
\newtheorem{thm}{Theorem}[section]
\newtheorem{conj}[thm]{Conjecture}

\newtheorem{lem}[thm]{Lemma}

\newtheorem{prop}[thm]{Proposition}
\newtheorem{prob}[thm]{Problem}

\allowdisplaybreaks
\newcommand{\tr}{\operatorname{tr}}

\newcommand{\fix}{\operatorname{fix}}

\begin{document}
\begin{center}
	{\bf \Large Newton polytopes of immanants of some combinatorial matrices}
\end{center}

\begin{center}
	{Candice X.T. Zhang      \\[6pt]}
	
	{\it~ Academy of Mathematics and Systems Science, \\
		Chinese Academy of Sciences, Beijing 100190, P. R. China\\[8pt]
		
		Email: {\tt zhangxutong@amss.ac.cn}}
\end{center}
\noindent\textbf{Abstract.}
The immanants of combinatorial matrices have many significant properties, including $m$-positivity and Schur positivity.
While the immanants of Jacobi-Trudi matrices are known to be both $m$-positive and Schur positive, those of Giambelli matrices have only been proven to be $m$-positive, with Schur positivity remaining a conjecture. These positivity properties rely heavily on lattice path interpretations.
In this paper, we study the Newton polytopes of immanants for these two classes of matrices. Using the lattice path method, we verify the saturation property for the Newton polytopes of Jacobi-Trudi matrices in special cases. For Giambelli matrices, we prove that this property holds for all immanants. To achieve this, we obtain the explicit coefficients of the largest monomial (in the dominance order) in the monomial expansion of the immanants of Giambelli matrices.



\section{Introduction}\label{sec-intro}

Given a matrix $A=\left(a_{i,j} \right)_{i,j=1}^{n}$, and a partition $\nu\vdash n$, the immanant of $A$ with respect to $\nu$ is defined by
\[{\rm Imm}_\nu A=\sum_{\pi\in\mathfrak{S}_{n}}\chi^{\nu}(\pi)\prod_{i=1}^{n} a_{i,\pi(i)},\]
where $\chi^{\nu}$ denotes the irreducible character of the symmetric group $\mathfrak{S}_{\ell}$ associated with the partition $\nu$ (refer to Section~\ref{sec-pre} or Fulton's book~\cite{Fulton} for the details).
This notion was introduced by Littlewood and Richardson~\cite{Littlewood} as a natural generalization of determinants (when $\nu=(1,1,\ldots,1)$) and permanents (when $\nu=(n)$), and have since then become a rich source of connections between representation theory and combinatorics.

The study of immanants of combinatorial matrices was initiated by Goulden and Jackson~\cite{GJ92}, who proved positivity properties for Jacobi-Trudi matrices. 
For a skew partition $\lambda/\mu$, the Jacobi-Trudi matrix is defined as
\[H(\lambda,\mu) = (h_{(\lambda_i-i)-(\mu_j-j)})_{i,j=1}^n,\] 
where $h_k$ denotes the $k$-th complete homogeneous symmetric function (see~\eqref{eq-def-h}). It is well-known that every skew Schur function (refer to~\eqref{eq-s-m}) can be expressed as the determinant of such a matrix:
\begin{equation}\label{eq-JT-identity}
	s_{\lambda/ \mu}=\det\left(H(\lambda,\mu) \right).
\end{equation}
Building on this determinantal formula, Goulden and Jackson initiated the study of immanants of Jacobi-Trudi matrices and proved that all such immanants expand with nonnegative coefficients in the monomial symmetric basis (see~\eqref{eq-def-monomial}), a property known as $m$-positivity. Later, Greene~\cite{Gre92} extended this result to arbitrary immanants, and Stembridge~\cite{Stembridge-1991} conjectured that these immanants are even Schur-positive. This conjecture was confirmed by Haiman~\cite{Hai93} using deep results from Kazhdan-Lusztig theory. 

For Giambelli matrices, the situation is different. Let $\lambda$ be a partition of rank $k$ with Frobenius notation $(\alpha\mid \beta)=(\alpha_1,\ldots,\alpha_k \mid \beta_1,\ldots,\beta_k)$, where $\alpha_i = \lambda_i - i$ and $\beta_i = \lambda'_i - i$ (see~\eqref{eq-Frobenius} and the definitions introduced before it). The associated Giambelli matrix is defined as
\[
G_{\lambda} = \bigl( s_{(\alpha_i \mid \beta_j)} \bigr)_{i,j=1}^{k}.
\]
Giambelli matrices also have deep connections with Schur functions as 
\begin{equation}\label{eq-Giambelli-identity}
	s_{\lambda}=\det \left( G_{\lambda}\right),
\end{equation}
refer to~\cite{ER-Giam} for a combinatorial proof. Using a planar network construction together with the Gessel-Viennot technology~\cite{Gessel-Viennot-1985}, Hamel and Goulden~\cite{HamelGoulden} unified four determinant formulas~\cite{ER-Giam,LP-1988,FK1997}, including the Jacobi-Trudi identity~\eqref{eq-JT-identity}, the dual Jacobi-Trudi identity, the Giambelli identity~\eqref{eq-Giambelli-identity}, and the Lascoux-Pragacz identity.
Based on the construction of cutting strips by Chen, Yan and Yang~\cite{CYY2004} (see Figure~\ref{fig-skew-partition} for an example) together with Wolfgang's criterion~\cite{Wolfgang} for immanants of matrices with $D$-compatible lattice paths, Li~\cite{LiPhD} proved that the immanants of four classes of Hamel--Goulden matrices are $m$-positive. However, their Schur positivity remains an open problem~\cite{LiPhD}.

In this paper, we study the Newton polytopes of immanants of these two families of matrices.
From now on, all symmetric functions can be viewed as symmetric polynomials by restricting in finitely many variables and we usually omit the variables.
For a real polynomial $f = \sum_{\alpha \in \mathbb{N}^n} c_\alpha \mathbf{x}^\alpha \in \mathbb{R}[x_1, \ldots, x_n]$, its support is
\[
\mathrm{supp}(f) = \{ \alpha \in \mathbb{N}^n \mid c_\alpha \ne 0 \},
\]
and its Newton polytope is
\[
\mathrm{Newton}(f) = \mathrm{conv}\{ \alpha \mid \alpha \in \mathrm{supp}(f) \} \subseteq \mathbb{R}^n.
\]
We say $f$ has a saturated Newton polytope, or is SNP for short, if
\[
\mathrm{supp}(f) = \mathrm{Newton}(f) \cap \mathbb{Z}^n.
\]
Many combinatorial polynomials are known to have this property, including (skew) Schur polynomials~\cite{Rado}, Schubert polynomials~\cite{FMD2018}, Stanley symmetric polynomials~\cite{MTY2017}, and dual $k$-Schur polynomials~\cite{WZZ2025}; see~\cite{MTY2017,WZZ2025} for further background.

For the immanants of Jacobi-Trudi matrices and Giambelli matrices, their Newton polytopes turn out to be $\lambda$-permutohedra. In particular, some of them can also be viewed as the dilation of simplices. Recall that the $\lambda$-permutahedron $\mathcal{P}_\lambda$ is the convex hull of the $\mathfrak{S}_n$-orbit of a partition $\lambda = (\lambda_1, \ldots, \lambda_n)$.
A simplex is the convex hull of $\{ u_i : 1\le i\le n\}$ (where $u_i$ are standard basis vectors). 
To be noted that $\mathcal{P}_{(d)}$ is the $d$-dilation of the simplex of the same dimension.
In addition, each $\lambda$-permutahedron can be viewed as a generalized permutahedron, which is a deformation of the usual permutahedron $\Pi_n$. Here a usual permutahedron is defined as the convex hull of all permutations of $(1,2,\dots,n)$, i.e., $\Pi_n = \operatorname{conv}\{(\pi(1),\pi(2),\dots,\pi(n)) \mid \pi \in \mathfrak{S}_n\}$; see for instance \cite{POS2009}.
Monical et al.~\cite{MTY2017} observed that every Schur polynomial is SNP and that its Newton polytope is a $\lambda$-permutahedron. Their observation is based on a result of Rado \cite[Theorem 1]{Rado}, which relates permutahedra to the dominance order.
\begin{thm}[{\cite[Theorem 1]{Rado}}]\label{thm-Rado-Schur}
	Let $\lambda$ and $\mu$ be two partitions of $d$. Then
	\[
	\mathcal{P}_{\mu} \subseteq \mathcal{P}_{\lambda}\text{ if and only if } \mu \unlhd \lambda.
	\]
\end{thm}

Based on this result, Monical, Tokcan, and Yong~\cite{MTY2017} provided a sufficient condition for a linear combination of Schur polynomials
 to have a saturated Newton polytope.
\begin{thm}[{\cite[Proposition 2.5]{MTY2017}}]\label{thm-SNP-Schur-combin}
	Let $f = \sum_{\mu} c_{\mu} s_{\mu}$ be a homogeneous symmetric polynomial of degree $d$. If there exists a partition $\lambda$ such that $c_{\lambda} \neq 0$ and $c_{\mu} \neq 0$ only if $\mu \unlhd \lambda$, then $\mathrm{Newton}(f) = \mathcal{P}_{\lambda}$. Moreover, if $c_{\mu} \geq 0$ for all $\mu$, then $f$ is SNP.
\end{thm}

Since the immanants of Jacobi--Trudi matrices generalize skew Schur functions and are known to be Schur positive, it is natural to ask whether they also satisfy the SNP property. For $\nu = (1,\dots,1)$, the immanant reduces to the determinant, so the property holds due to the Jacobi--Trudi identity \eqref{eq-JT-identity} and the fact that every skew Schur polynomial is SNP~\cite{Rado,MTY2017}. For the cases $\nu = (n)$ and $\nu = (n-1,1)$ (the latter under the additional condition that $\lambda/\mu$ is a border strip, i.e., a connected skew shape with no $2\times2$ square), we obtain the following result.

\begin{thm}\label{thm-SNP-JT}
	Given a skew partition $\lambda/\mu$ of $d$. For $\nu=(n-1,1)$, the immanant of Jacobi-Trudi matrix ${\rm Imm}_\nu H(\lambda,\mu)$ is SNP if $\lambda/\mu$ is a border strip. For $\nu=(n)$, the polynomial ${\rm Imm}_\nu H(\lambda,\mu)$ is SNP in general. In particular, in both cases, ${\rm Newton}({\rm Imm}_\nu H(\lambda,\mu))$ is a $d$-dilation of simplex.
\end{thm}

In the study of Schur positivity of immanants of Jacobi-Trudi matrices, Stanley and Stembridge~\cite{Stanley-Stembridge-1993} introduced the symmetric functions
\[
F_{\lambda/\mu}(\mathbf{x},\mathbf{y}) = \sum_{\nu \vdash n} s_{\nu}(\mathbf{y})\,\operatorname{Imm}_{\nu} H(\lambda,\mu)(\mathbf{x}),
\]
and defined $E_{\lambda/\mu}^{\theta}(\mathbf{y})$ as the coefficient of $s_{\theta}(\mathbf{x})$ in $F_{\lambda/\mu}(\mathbf{x},\mathbf{y})$. That is,
\begin{equation}\label{eq-E-def}
	F_{\lambda/\mu}(\mathbf{x},\mathbf{y}) = \sum_{\theta \vdash d} E_{\lambda/\mu}^{\theta}(\mathbf{y})\, s_{\theta}(\mathbf{x}),
\end{equation}
where $n$ and $d$ denote the length and the size of $\lambda/\mu$, respectively. When $\lambda/\mu$ is a border strip, these $E$-polynomials have a simple $p$-expansion (refer to Proposition~\ref{prop-E-expression}). Building on this, we find that each $E_{\lambda/\mu}^{\theta}$ in this case is SNP.

\begin{thm}\label{thm-E-SNP}
	Given a skew partition $\lambda/\mu$ and a partition $\theta$ of size $d$. If $\lambda/\mu$ is a border strip, then the polynomial $E_{\lambda/\mu}^{\theta}$ is SNP. In particular, its Newton polytope is a $d$-dilation of simplex.
\end{thm}

Similar to Jacobi-Trudi matrices, Giambelli matrices also have connections with Schur polynomials, as given in~\eqref{eq-Giambelli-identity}.
By applying Hamel-Goulden's planar network construction~\cite{HamelGoulden}, in this paper, we further establish the following result concerning the Newton polytopes of the immanants of Giambelli matrices.

\begin{thm}\label{thm-SNP-Giam}
	Let $\lambda \vdash d$ be a partition of rank $k$. Then for any partition $\nu \vdash k$, the immanant $\operatorname{Imm}_{\nu} G_{\lambda}$ is SNP and its Newton polytope is a $\lambda$-permutahedron.
\end{thm}

The remainder of this paper is organized as follows.
In Section~\ref{sec-pre}, we provide relevant definitions and preliminary results.
In Section~\ref{sec-JT}, we recall Greene's planar network construction for immanants of Jacobi--Trudi matrices and present the proof of Theorem~\ref{thm-SNP-JT}. 
Additionally, we prove Theorem~\ref{thm-E-SNP} building on a refinement result of Kostka numbers.
In Section~\ref{sec-Gm}, we first recall Hamel and Goulden's planar network construction and then give the proof of Theorem~\ref{thm-SNP-Giam}.
Furthermore, we provide an explicit formula for the coefficient of the leading monomial in the monomial expansion of the immanant of each Giambelli matrix. Finally, in Section~\ref{sec-further-work}, we discuss some open problems related to Lorentzian polynomials for further study.

\section{Preliminaries}\label{sec-pre}

In this section, we recall basic definitions and known results on symmetric functions and representation theory; see Stanley~\cite{StaEC2} and Fulton~\cite{Fulton} for details.

We begin by recalling the definitions of symmetric functions. Given a partition $\lambda$, we identify it with the infinite sequence $(\lambda_1,\ldots,\lambda_n,0,\ldots,0)$. The monomial symmetric function $m_\lambda$ is defined as
\begin{equation}\label{eq-def-monomial}
	m_\lambda = \sum_{\alpha} \mathbf{x}^\alpha,
\end{equation}
where the sum runs over all distinct permutations $\alpha$ of $\lambda$, and $\mathbf{x}^\alpha = x_1^{\alpha_1}x_2^{\alpha_2}\cdots$.
Given a partition $\lambda = (\lambda_1,\ldots,\lambda_n)$, 
the power sum symmetric function is defined by $p_\lambda = p_{\lambda_1}\cdots p_{\lambda_n}$, where for $k\ge 1$, 
\[p_k =m_k= \sum_i x_i^k,\]  
The {Schur function} $s_{\lambda}$ can be defined via its expansion into the monomial symmetric functions, namely,
\begin{equation}\label{eq-s-m}
	s_{\lambda} = \sum_{\mu} K_{\lambda,\mu} \, m_{\mu}.
\end{equation}
Here, $K_{\lambda,\mu}$ denotes the {Kostka number}, which counts the number of semistandard Young tableaux of shape $\lambda$ and content $\mu$.
The complete homogeneous symmetric function $h_{\lambda}$ is defined by $h_\lambda = h_{\lambda_1}\cdots h_{\lambda_n}$, where
\begin{equation}\label{eq-def-h}
	h_k = \sum_{\nu \vdash k} m_{\nu}=\sum_{i_1\le\cdots\le i_k} x_{i_1}\cdots x_{i_k},
\end{equation}
with the conventions $h_0 = 1$ and $h_k = 0$ for $k<0$.
It is well-known that these functions can be expressed in terms of Schur functions as
\begin{equation}\label{eq-h-s}
	h_{\lambda} = \sum_{\mu} K_{\mu,\lambda} \, s_{\mu}.
\end{equation}
It is important to note that the non-vanishing of terms in both expansions \eqref{eq-s-m} and \eqref{eq-h-s} is governed by the dominance order on partitions. Given two partitions $\lambda$ and $\mu$ of $d$, we say $\mu$ is dominated by $\lambda$, denoted $\mu \unlhd \lambda$, if
\[
\mu_1 + \cdots + \mu_i \leq \lambda_1 + \cdots + \lambda_i \quad \text{for all } i \geq 1.
\]
The following well-known theorem characterizes the non-vanishing of Kostka numbers in terms of this order.
\begin{thm}[{\cite[Proposition 7.10.5 and Exercise 7.12]{StaEC2}}]\label{thm-Kostka-dominance}
	Let $\lambda$ and $\mu$ be partitions of $d$. Then the Kostka number $K_{\lambda,\mu}$ is non-zero if and only if $\mu \unlhd \lambda$. In particular, $K_{\lambda,\lambda}=1$.
	\end{thm}

Next, turn to the introduction of representation theory, refer to \cite{Fulton,Sagan} for details. Given a finite group $G$, we call a representation of $G$ on a finite-dimensional real vector space $V$ is a homomorphism $\rho:G\rightarrow {\rm GL}(V)$. 
Naturally, the trivial representation of $G$ is the one-dimensional representation $(\rho_{\text{triv}}, V)$ defined by
\[
\rho_{\text{triv}}(g) = \operatorname{id}_V \quad \text{for all } g \in G.
\]
We sometimes call $V$ itself a representation of $G$. In this manner, if $V$ is a representation of $G$, the character of $V$, denoted by $\chi_V$, is the real-valued function on the group defined by $\chi_V(g)=\tr (g|_V)$, where $\tr$ denote the trace. Note that all characters are class functions, namely, they satisfy $\chi_V(hgh^{-1}) = \chi_V(g)$ for all $g, h \in G$. We call a representation $(\rho, V)$ is irreducible if $V$ contains no non-trivial proper $G$-invariant subspaces, and a character $\chi$ is said to be irreducible if it is the character of an irreducible representation. 
It is clear that the character of any trivial representation is
\[
	\chi^{\text{triv}}(g) = 1 \quad \text{for all } g \in G.
\]

In this paper, we focus on the case where $G$ is the symmetric group $\mathfrak{S}_n$. In this case, the irreducible characters of $\mathfrak{S}_n$ are in one-to-one correspondence with the partitions $\lambda$ of $n$. We denote the character associated to $\lambda$ by $\chi^\lambda$. They form an orthonormal basis under the inner product:
\[
\langle \chi, \psi \rangle = \frac{1}{|G|} \sum_{g \in G} \chi(g)\overline{\psi(g)}.
\]
For the special partition $(n-1,1)$, the corresponding representation is called the standard representation of $\mathfrak{S}_n$. It is realized on the subspace:
\[
W = \left\{ (x_1, \dots, x_n) \in \mathbb{R}^n\mid \sum_{i=1}^n x_i = 0 \right\},
\]
where $\mathfrak{S}_n$ acts by permuting coordinates. 
It is well-known that the irreducible character of the standard representation is given by
\begin{equation}\label{eq-standard}
	\chi^{(n-1,1)}(\pi) = \fix(\pi) - 1,
\end{equation}
where $\fix(\pi)$ is the number of fixed points of the permutation $\pi$, refer to \cite{Fulton} for the details.
Besides, we also need to consider a class of subgroup of $\mathfrak{S}_n$, called Young subgroup.
Given a partition $\lambda = (\lambda_1, \lambda_2, \dots, \lambda_{\ell})$ of $n$, the corresponding Young subgroup $\mathfrak{S}_\lambda$ of the symmetric group $\mathfrak{S}_n$ is defined as
\[
	\mathfrak{S}_\lambda = \mathfrak{S}_{\{1, \dots, \lambda_1\}} \times \mathfrak{S}_{\{\lambda_1+1, \dots, \lambda_1+\lambda_2\}} \times \dots \times \mathfrak{S}_{\{n-\lambda_{\ell}+1, \dots, n\}} \subset \mathfrak{S}_n.,
\]
and it is isomorphic to the direct product $\mathfrak{S}_{\lambda_1} \times \mathfrak{S}_{\lambda_2} \times \dots \times \mathfrak{S}_{\lambda_{\ell}}$. It is clear to see that 
\begin{equation}\label{eq-card-S_lam}
	|\mathfrak{S}_{\lambda}|= \lambda_1!\lambda_2!\cdots \lambda_{\ell}!.
\end{equation}
For the irreducible character of Young subgroup $\mathfrak{S}_{\lambda}$, we derive the following lemma.
\begin{lem}\label{lem-Young_char}
	Given a partition $\lambda=(\lambda_1,\lambda_2,\ldots,\lambda_{n})\vdash d$,
	we have 
	\[\chi^{\lambda}\left(\mathfrak{S}_\lambda \right)=\lambda_1!\lambda_2!\cdots\lambda_{n}! , \]
	and $\chi^{\lambda}(\mathfrak{S}_{\mu})=0$ for any $\mu\trianglerighteq \lambda$ in dominance order.
\end{lem}
\begin{proof}
	By definition, we have
	\[\chi^{\lambda}(\mathfrak{S}_{\mu})=|\mathfrak{S}_{\mu}|\cdot \langle \chi^{\lambda}|_{\mathfrak{S}_{\mu}}, \chi^{\rm triv} \rangle.\]
	Note that the inner product is preserved under the Frobenius map ${\rm ch}$, which maps class functions to symmetric functions, refer to \cite[Section 7.18]{StaEC2} for the details. 
	For the group $\mathfrak{S}_n$, 
	\[{\rm ch}(\chi^{\rm triv})=h_d, \text{ and } {\rm ch}(\chi^{\lambda})=\mathfrak{S}_{\lambda}.\]
	By combining the equation~\eqref{eq-h-s}, we can deduce that
	\[\langle \chi^{\lambda}|_{\mathfrak{S}_{\mu}}, \chi^{\rm triv} \rangle =\langle \mathfrak{S}_{\mu}, h_{\lambda}\rangle =K_{\lambda,\mu}.\]
	Then the lemma holds by applying Theorem~\ref{thm-Kostka-dominance} and Equation~\eqref{eq-card-S_lam}.
\end{proof}

Finally, we need to introduce the basic definitions of outside decompositions, cutting strips, and other relevant notations.
Given a skew partition $\lambda/\mu$ of length $n$, we now identify it with a skew diagram defined as the set of boxes $\{(i,j) \in \mathbb{Z}^2 \mid \mu_i < j \le \lambda_i \text{ and } 1 \le i \le n\}$ (using the English convention). The conjugate of $\lambda/\mu$, denoted $(\lambda/\mu)'$ (or $\lambda'/\mu'$), is the skew diagram obtained by transposing this set of boxes across the main diagonal.
When $\mu=\emptyset$, the rank of $\lambda$ is defined as the integer $k$ such that $(k,k) \in \lambda$ and $(k+1,k+1) \notin \lambda$.
Given a partition $\lambda$ of rank $k$, for $1 \le i \le k$, let $\alpha_i$ denote the number of cells to the right of $(i,i)$, and $\beta_i$ the number of cells below $(i,i)$.
Then the Frobenius notation of $\lambda$ is 
\begin{equation}\label{eq-Frobenius}
	(\alpha \mid \beta) = (\alpha_1, \ldots, \alpha_k \mid \beta_1, \ldots, \beta_k).
\end{equation}
Given a skew diagram $\lambda/\mu$, the content of a box $(i,j)$ is defined as $c = j - i$, and the $c$-th diagonal consists of all boxes in $\lambda/\mu$ having content $c$. Then a (border) strip can be also defined as a connected skew diagram whose cells have distinct contents.

An outside decomposition of $\lambda/\mu$ is a partition of the boxes of $\lambda/\mu$ into pairwise disjoint strips such that every strip in the decomposition has a starting box on the left or bottom perimeter of the diagram and an ending box on the right or top perimeter of the diagram. 
In a strip, any two adjacent boxes share exactly one common edge, which naturally gives rise to a well-defined direction from the starting box to the ending box. Specifically, for each box that is not the ending box, we say that it goes right if the next box is to its right, otherwise, we say that it goes up. 
Based on Hamel and Goulden's observation~\cite{HamelGoulden}, the contents of the ending (or starting) boxes of the strips in an outside decomposition are distinct. Thus, we may use the contents of the ending boxes to order the strips in an outside decomposition.
For an outside decomposition $\Pi$ with $\ell$ strips, we denote it by $\Pi=(\theta_1,\theta_2,\ldots,\theta_{\ell})$, where $\theta_1$ is the strip whose ending box has the largest content. 
For example, let $\lambda/\mu = (6,5,5,4,3)/(3,3,1,1)$, an outside decomposition 
\begin{equation}\label{eq-exam-outside}
	\Pi=((3),(2),(2,1),(2,2)/(1),(1),(3))
\end{equation} 
of this skew shape is depicted on the left side of Figure~\ref{fig-skew-partition}, and the directions of boxes in $\theta_3$ are: up, right, right.

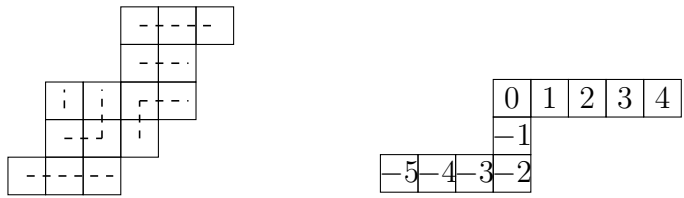
\begin{figure}[ht]
	\centering

	\begin{tikzpicture}[scale=0.5]

		
		\foreach \x in {3,...,5} {
			\draw (\x, -1) rectangle (\x+1, 0);
		}
		
		\foreach \x in {3,...,4} {
			\draw (\x, -2) rectangle (\x+1, -1);
		}
		
		\foreach \x in {1,...,4} {
			\draw (\x, -3) rectangle (\x+1, -2);
		}
		
		\foreach \x in {1,...,3} {
			\draw (\x, -4) rectangle (\x+1, -3);
		}
		
		\foreach \x in {0,...,2} {
			\draw (\x, -5) rectangle (\x+1, -4);
		}

		\draw[dashed, line width=0.6pt] (0.5, -4.5) -- (2.8, -4.5);    
		\draw[dashed, line width=0.6pt] (1.5, -3.5) -- (2.5, -3.5);  
		\draw[dashed, line width=0.6pt] (2.5, -3.5) -- (2.5, -2.2);   
		
		\draw[dashed, line width=0.6pt] (1.5, -2.7) -- (1.5, -2.2); 
		\draw[dashed, line width=0.6pt] (3.5, -3.5) -- (3.5, -2.5); 
		\draw[dashed, line width=0.6pt] (3.5, -2.5) -- (4.8, -2.5); 
		\draw[dashed, line width=0.6pt] (3.5, -1.5) -- (4.8, -1.5);
		\draw[dashed, line width=0.6pt] (3.5, -0.5) -- (5.5, -0.5);
		
	\end{tikzpicture}
	\qquad\qquad
		\begin{tikzpicture}[scale=0.5]
		
		
		\foreach \x in {3,...,7} {
			\draw (\x, 0) rectangle (\x+1, 1);
		}
		
		\foreach \x in {3,...,3} {
			\draw (\x, -1) rectangle (\x+1, 0);
		}
		
		\foreach \x in {0,...,3} {
			\draw (\x, -2) rectangle (\x+1, -1);
		}
	
		\node at (7.5,0.5) {$4$};
		\node at (6.5,0.5) {$3$};
		\node at (5.5,0.5) {$2$};
		\node at (4.5,0.5) {$1$};
		\node at (3.5,0.5) {$0$};
		\node at (3.5,-0.5) {$-1$};
		\node at (3.5,-1.5) {$-2$};
		\node at (2.5,-1.5) {$-3$};
		\node at (1.5,-1.5) {$-4$};
		\node at (0.5,-1.5) {$-5$};
	\end{tikzpicture}
	\caption{An outside decomposition and the cutting strip (with contents) of $\lambda/\mu$.}
	\label{fig-skew-partition}
\end{figure}

Let $\Pi$ be an outside decomposition of $\lambda/\mu$, and suppose $\lambda/\mu$ has $\ell$ diagonals. The cutting strip $T$ of $\Pi$ is constructed as follows: for $1\le i\le \ell-1$, the $i$-th box in $T$ goes right or goes up, matching the direction of the boxes in the $i$-th diagonal of $\lambda/\mu$ with respect to the decomposition $\Pi$.
The corresponding cutting strip (with contents) for the outside decomposition in~\eqref{eq-exam-outside} is illustrated on the right side of Figure~\ref{fig-skew-partition}. Given an outside decomposition $\Pi$ of shape $\lambda/\mu$ and a strip in $\Pi$. Let $\phi$ be the corresponding cutting strip of $\Pi$. For $p\le q$, let $[p,q]$ denote a segment of the cutting strip $\phi$ starting with the box having content $p$ and ending with the box having content $q$.
Chen, Yan, and Yang~\cite{CYY2004} gave the correspondence of the strip $\theta_i \# \theta_j$ defined in \cite{HamelGoulden} with $[p(\theta_j),q(\theta_i)]$, where $p(\theta)$ (respectively $q(\theta)$) denotes the content of the starting box (respectively ending box) of a strip $\theta$.

For example, for the outside decomposition in~\eqref{eq-exam-outside},
some strips obtained by the operation $\#$ are given as follows:
\begin{equation*}
	\begin{array}{ll}
		\theta_1\#\theta_5=[-2,4]=(5,1,1), & \theta_5\#\theta_1=[2,-2]= \text{undefined} \\
		\theta_2\#\theta_3=[-1,2]=(3,1),         & \theta_3\#\theta_2=[1,1]=(1) \\
		\theta_4\#\theta_6=[-5,-1]=(4,4)/(3),    & \theta_6\#\theta_4=[-3,-3]=(1).
	\end{array}
\end{equation*}

For a partition $\lambda$ of rank $k$ with Frobenius notation $(\alpha\mid\beta)=(\alpha_1,\ldots,\alpha_k\mid\beta_1,\ldots,\beta_k)$, the Giambelli matrix $G_{\lambda}$ is a special case of the Hamel–Goulden matrix~\cite{HamelGoulden}. Indeed, for $1\le i,j\le k$ we have $\theta_i\#\theta_j = [j-\lambda'_j,\lambda_i-i] = (\alpha_i\mid\alpha_j)$, so the outside decomposition of $\lambda$ consists precisely of the hooks on the diagonal.
For example, let $\lambda = (6,6,4,4,1,1)$, as shown in Figure~\ref{fig-out-Gm}, the outside decomposition of $\lambda$ is $\Pi = (\theta_1, \theta_2, \theta_3, \theta_4)$, where
\[
\theta_1 = (6,1^5),\quad \theta_2 = (5,1^2),\quad \theta_3 = (2,1),\quad \theta_4 = (1),
\]
and the corresponding cutting strip is $(6,1^5)$ as shown on the right side.

\begin{figure}[ht]
	\centering
	\begin{tikzpicture}[scale=0.5]                                                                                           \foreach \y in {0,...,5} {
			\ifnum\y=0\def\cols{6}\fi
			\ifnum\y=1\def\cols{6}\fi
			\ifnum\y=2\def\cols{4}\fi    
			\ifnum\y=3\def\cols{4}\fi
			\ifnum\y=4\def\cols{1}\fi
			\ifnum\y=5\def\cols{1}\fi
			
			\foreach \x in {1,...,\cols} {
				\draw (\x, -\y) rectangle (\x+1, -\y-1);
			}
		}

		\draw[dashed, line width=0.4pt] (1.5, -0.5) -- (1.5, -5.7);   
		\draw[dashed, line width=0.4pt] (1.5, -0.5) -- (6.6, -0.5);   

		\draw[dashed, line width=0.4pt] (2.5, -1.5) -- (2.5, -3.8);  
		\draw[dashed, line width=0.4pt] (2.5, -1.5) -- (6.5, -1.5); 

		\draw[dashed, line width=0.4pt] (3.5, -2.5) -- (3.5, -3.7); 
		\draw[dashed, line width=0.4pt] (3.5, -2.5) -- (4.7, -2.5); 

		\draw[dashed, line width=0.4pt] (4.2, -3.5) -- (4.7, -3.5);
		
	\end{tikzpicture}
	\qquad\qquad
		\begin{tikzpicture}[scale=0.5]                                             \foreach \y in {0,...,5} {
			\ifnum\y=0\def\cols{6}\fi
			\ifnum\y=1\def\cols{1}\fi
			\ifnum\y=2\def\cols{1}\fi    
			\ifnum\y=3\def\cols{1}\fi
			\ifnum\y=4\def\cols{1}\fi
			\ifnum\y=5\def\cols{1}\fi
			
			\foreach \x in {1,...,\cols} {
				\draw (\x, -\y) rectangle (\x+1, -\y-1);
			}
		}

	\end{tikzpicture}
	\caption{An outside decomposition and the cutting strip of $\lambda=(6,6,4,4,1,1)$.}
	\label{fig-out-Gm}
\end{figure}
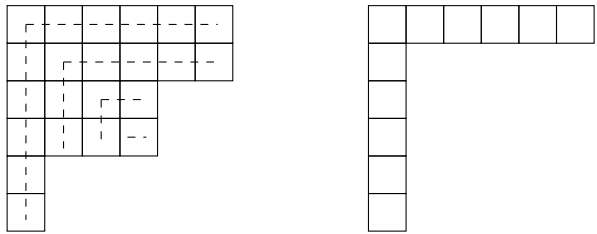

\section{Immanants of Jacobi-Trudi matrices}\label{sec-JT}
In this section, we first recall some related results concerning the immanants of Jacobi-Trudi matrices presented by Greene~\cite{Gre92}, including the lattice path construction of Gessel and Viennot~\cite{Gessel-Viennot-1985} and the reduction of Goulden and Jackson~\cite{GJ92}. Based on these, we provide a proof of Theorem~\ref{thm-SNP-JT}. Furthermore, we also prove the SNP property for the polynomials defined by Stanley and Stembridge~\cite{Stanley-Stembridge-1993} in the border strip case. 

Greene's lattice path construction~\cite{Gre92} is inherted from Gessel and Viennot's approach~\cite{Gessel-Viennot-1985}, which was used to give a combinatorial interpretation of the Jacobi-Trudi identity~\eqref{eq-JT-identity}.
The crucial idea is to interpret the complete homogeneous symmetric function $h_i(x_1, \ldots, x_{\ell})$ as the generating function for certain lattice paths. These paths start at $(0, 1)$ and end at $(i, \ell)$, where for $1\le j\le \ell$, a horizontal step at height $j$ is weighted by $x_j$, a vertical step is weighted by $1$, and the weight of a path is the product of the weights of its steps.
Then the product
\[\prod_{i=1}^n h_{(\lambda_i-i)-(\mu_{\pi_i}-\pi_i)}\]
can be interpreted as the generating function for $n$ (not necessarily disjoint) lattice paths, where the $i$-th path runs from the point $(\mu_{\pi_i} - \pi_i, 1)$ to the point $(\lambda_i - i, \ell)$. The weight of a path is the product of the weights of its steps, as defined previously.

For $i \in [n]$, let $P_i = (\mu_i - i, 1)$ and $Q_i = (\lambda_i - i, \ell)$, and let $\mathcal{F}$ be the collection of all lattice paths connecting ${P_1, \dots, P_n}$ to ${Q_1, \dots, Q_n}$ according to some permutation $\pi\in\mathfrak{S}_n$. Greene~\cite{Gre92} defined the skeleton of $\mathcal{F}$ as
\[\alpha=\alpha_{\mathcal{F}}=\bigcup_{P\in \mathcal{F}} \{\text{multiset of edges in some path }P\}.\]
In the same manner, let ${\bf x}^{\alpha}=\prod_{P\in\mathcal{F}}{\bf x}^{{\rm weight}(P)}$. Then each immanant of Jacobi-Trudi matrix has the following expansion:
\begin{align}
	{\rm Imm}_\nu H(\lambda,\mu) &= \sum_{\pi\in \mathfrak{S}_{n}}\chi^{\nu}(\pi)\prod_{i=1}^{n}h_{(\lambda_i-i)-(\mu_{\pi_i}-\pi_i)} \nonumber \\
	&= \sum_{\pi\in\mathfrak{S}_n}\chi^{\nu}(\pi)\sum_{\mathcal{F}}{\bf x}^{\alpha_{\mathcal{F}}} \nonumber \\
	&= \sum_{\alpha}{\bf x}^{\alpha}\sum_{\mathcal{F}}\chi^{\nu}(\pi_{\mathcal{F}}) \nonumber \\
	&= \sum_{\alpha}{\bf x}^{\alpha}\chi^{\nu}\left(\sum_{\mathcal{F}}\pi_{\mathcal{F}}\right), \label{eq-imm-expression}
\end{align}
where the inner sum is over all families of paths $\mathcal{F}$ having skeleton $\alpha$, and $\pi_{\mathcal{F}}$ denotes the permutation of endpoints induced by $\mathcal{F}$.

Goulden and Jackson showed the following result which was also noted in Greene's paper.
\begin{prop}[\cite{GJ92}]\label{prop-skleton}
	For lattice paths with fixed starting points $\mathcal{P}=\{P_1,P_2,\ldots,P_n\}$ and endpoints $\mathcal{Q}=\{Q_1,Q_2\ldots,Q_n\}$, there exists some subintervals $J_1,J_2,\ldots,J_{\ell}$ of $[n]$ such that $\sum_{\mathcal{F}}\pi_{\mathcal{F}}=\mathfrak{S}_{J_1}\times \mathfrak{S}_{J_1}\times\cdots \times \mathfrak{S}_{J_{\ell}}$, where $\mathcal{F}$ ranges over all lattice paths from $\mathcal{P}$ to $\mathcal{Q}$.  
\end{prop}

Now we are prepared to present the proof of Theorem~\ref{thm-SNP-JT}.

\begin{proof}[Proof of Theorem~\ref{thm-SNP-JT}]
	
	Given partitions $\lambda=(\lambda_{1},\ldots,\lambda_n)$ and $\mu=(\mu_1,\ldots,\mu_n)$ with $\mu_i\le \lambda_i$,
	it is clear that the degree of ${\rm Imm}_{\nu}H(\lambda,\mu)$ is $d=|\lambda_1-1|+|\mu_n-n|$.
	
	By definition, for the skeleton $\alpha = (d)$, the lattice family $\mathcal{F}$ corresponding to $\alpha$ must have the form depicted in Figure~\ref{fig-d2l}.
	\begin{figure}[ht]
		\centering
		\begin{tikzpicture}
			[place/.style={thick,fill=black!100,circle,inner sep=0pt,minimum size=1mm,draw=black!100},scale=1.2]
			\node [place,label=below:{\footnotesize$P_n$}] at (0,0) {};
			\node [place,label=below:{\footnotesize$P_{n-1}$}] at (1,0) {};
			\node [place,label=below:{\footnotesize$P_{n-2}$}] at (2,0) {};
			\node [label=below:{\footnotesize$\cdots$}] at (3.5,0) {};
			\node [place,label=below:{\footnotesize$P_1$}] at (4.5,0) {};
			\node [place,label=above:{\footnotesize$Q_n$}] at (1,2) {};
			\node [place,label=above:{\footnotesize$Q_{n-1}$}] at (2,2) {};
			\node [label=above:{\footnotesize$\cdots$}] at (3.5,2) {};
			\node [place,label=above:{\footnotesize$Q_2$}] at (4.5,2) {};
			\node [place,label=above:{\footnotesize$Q_1$}] at (5.5,2) {};
			\node [place] at (0,1) {};
			\node [place] at (1,1) {};
			\node [place] at (2,1) {};
			\node [place] at (4.5,1) {};
			\node [place] at (5.5,1) {};
			\node [label=above:{\footnotesize$x_j$}] at (0.5,1) {};
			\node [label=above:{\footnotesize$x_j$}] at (1.5,1) {};
			\node [label=above:{\footnotesize$\cdots$}] at (3.5,1) {};
			\node [label=above:{\footnotesize$x_j$}] at (5,1) {};
			\draw [thick]  (0,0) --(0,1);
			\draw [thick]  (0,1) --(1,1);
			\draw [thick]  (1,1) --(2,1);
			\draw [thick]  (2,1) --(4.5,1);
			\draw [thick]  (4.5,1) --(5.5,1);
			\draw [thick]  (1,0) --(1,2);
			\draw [thick]  (2,0) --(2,2);
			\draw [thick]  (4.5,0) --(4.5,2);
			\draw [thick]  (5.5,1) --(5.5,2);
		\end{tikzpicture}
		\caption{The skeleton $\alpha=(d)$ with weight ${\bf x}^{\alpha}=x_j^d$}\label{fig-d2l}
	\end{figure}
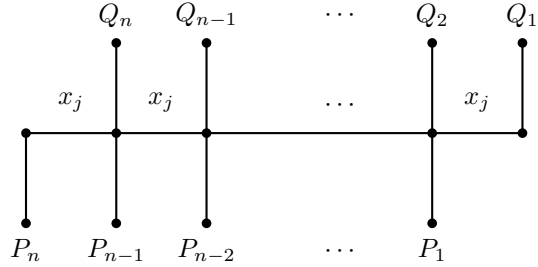
    
    Therefore, by applying Proposition~\ref{prop-skleton}, we have
	\begin{align*}
		\chi^{\nu}(S_J)=\chi^{\nu}\left(\mathfrak{S}_{\{1,2\}}\times \mathfrak{S}_{\{2,3\}}\times \cdots\times \mathfrak{S}_{\{n-1,n\}}\right)
	\end{align*}
	If we set $\nu=(n-1,1)$, then $V_{\nu}$ is the standard representation
	\[V=\{(x_1,\ldots,x_n)\mid x_1+x_2+\cdots+x_n=0\},\]
	and as stated in \eqref{eq-standard}, for each $g\in \mathfrak{S}_n$, its irreducible character has the explicit explanation
	\[\chi_V(g)=\fix(g)-1. \]
	
	Let $f_k(n)$ denote the number of permutations in $\mathfrak{S}_{\{1,2\}}\times \mathfrak{S}_{\{2,3\}}\times\cdots\times \mathfrak{S}_{\{n-1,n\}}$ with $k$ fixed points.
	Then we have 
	\[\chi^{(n-1,1)}\left(\mathfrak{S}_{\{1,2\}}\times \mathfrak{S}_{\{2,3\}}\times \cdots\times \mathfrak{S}_{\{n-1,n\}} \right) =\sum_{k=0}^nf_k(n)\cdot (k-1).\]
	For all $n\ge 1$, we have $f_k(n)\ge 0$ and $f_n(n)=1$. Moreover, for $n\ge 3$, we have $f_0(n)\ge n-2$ since the permutation $(123\cdots n)$ and those permutations of the form $(1\cdots i)(i+1\cdots n)$ (of which there are $n-3$ choices for $i$) all have zero fixed points. Thus, the coefficient of $x_j^d$ must be non-zero for each $1\le j\le k$. 
	Specifically, let $[x_{\alpha}]f$ denote the coefficient of the monomial $x_{\alpha}$ in $f$, by the expression illustrated in \eqref{eq-imm-expression}, we have
	\begin{align*}
		[x_j^d]\,{\rm Imm}_{(n-1,1)}(\lambda,\mu)&=\chi^{(n-1,1)}\left(\mathfrak{S}_{\{1,2\}}\mathfrak{S}_{\{2,3\}}\cdots \mathfrak{S}_{\{n-1,n\}}\right)\\
		&\ge(n-2)(-1)+\sum_{k=1}^{n-1}f_k(n)(k-1)+n-1>0
	\end{align*}
	Based on the symmetric property, the Schur expansion of ${\rm Imm}^{(n-1,1)}H(\lambda,\mu)$ must have $s_{(d)}$ as the maximal term on dominance order if $\lambda/\mu$ is a border strip of size $d$.
	Consequently, 
	\[{\rm Newton}\left({\rm Imm}_{(n-1,1)}H(\lambda,\mu) \right) =\mathcal{P}_{(d)},\]
	which follows by combining Theorem~\ref{thm-SNP-Schur-combin} with the Schur-positivity of ${\rm Imm}^{(n-1,1)}H(\lambda,\mu)$.
	
	Now, for $\nu=(n)$, the immanant specializes to the permanent, and we have
	\[{\rm supp}\left( {\rm per}\, H(\lambda,\mu)\right)= {\rm supp}\left(s_{(d)} \right)\]
	for any skew partition $\lambda/\mu \vdash d$.
	This follows from Equation \eqref{eq-h-s}, Theorem \ref{thm-Rado-Schur} and the nonnegativity of those coefficients in the $h$-expansion of ${\rm per}\, H(\lambda,\mu)$.
	This completes the proof of Theorem~\ref{thm-SNP-JT}.
	
\end{proof}

Next, to illustrate the second main result of this section, we first recall some basic notations.
Given a composition $\alpha=(\alpha_1,\ldots,\alpha_{\ell})$ of positive integer $d$.
We call a refinement of composition $\alpha$ is the replacement of $\alpha_i$ by a composition of $\alpha_i$. 
If $\beta$ is a refinement of $\alpha$, then conversely, $\alpha$ is a meld of $\beta$. 
Note that the compositions of $\ell$ (the length of $\alpha$) can be used to index the melds of $\alpha$.
Let $\alpha|_{\gamma}$ denote the meld of $\alpha$ indexed by $\gamma$.

Recall that $E_{\lambda/\mu}^{\theta}({\bf y})$ is the coefficient of $s_{\theta}({\bf x})$ in $F_{\lambda/ \mu}({\bf x},{\bf y})$ (refer to ).
Stanley and Stembridge~\cite{Stanley-Stembridge-1993} gave the $p$-expansion of $E_{\lambda/\mu}^{\theta}$ for the case when $\lambda/\mu$ is a border strip.
\begin{prop}[{\cite[Lemma 2.1]{Stanley-Stembridge-1993}}]\label{prop-E-expression}
	If $\lambda/\mu$ is a border strip, then
	\[E_{\lambda/\mu}^{\theta}=\sum_{\alpha}K_{\theta,\,(\lambda-\mu)|_{\alpha}}p_{\alpha_1}\cdots p_{\alpha_{\ell}}, \]
	where the summation is over all compositions of $n$, i.e. the length of $\lambda/\mu$.
\end{prop}

We find that the Kostka numbers satisfy the following refinement condition.
\begin{prop}\label{prop-refine-Kostka}
	Let $\lambda/\mu$ and $\theta$ be two partitions of the same size. 
	If $K_{\theta,\,(\lambda-\mu)|_{\alpha}}\neq 0$, and $\beta$ is a refinement of $\alpha$, then we have $K_{\theta,\,(\lambda-\mu)|_{\beta}}\neq 0$.
\end{prop}
\begin{proof}
	Given a tableaux of shape $\theta$, and type $(\lambda-\mu)|_{\alpha}$. Set $\lambda-\mu=\gamma$, then by definition we have
	\begin{align*}
		\gamma|_{\alpha}=(\gamma_1+\cdots+\gamma_{\alpha_1},\gamma_{\alpha_1+1}+\cdots+\gamma_{\alpha_1+\alpha_2},\ldots)\\
		\gamma|_{\beta}=(\gamma_1+\cdots+\gamma_{\beta_1},\gamma_{\beta_1+1}+\cdots+\gamma_{\beta_1+\beta_2},\ldots)
	\end{align*}
	Since $\beta$ is a refinement of $\alpha$, we have $\gamma|_{\beta}\trianglelefteq \gamma|_{\alpha}$ in dominance order.
	Thus, by applying Theorem~\ref{thm-Kostka-dominance}, $K_{\theta,\,\gamma|_{\alpha}}\neq 0$ implies that $\gamma|_{\beta}\trianglelefteq \gamma|_{\alpha}\trianglelefteq \theta$, and therefore $K_{\theta, \, \gamma|_{\beta}}\neq 0$. This completes the proof.
\end{proof}

Combining the above results, we can finally present the proof of Theorem~\ref{thm-E-SNP}.
\begin{proof}[Proof of Theorem~\ref{thm-E-SNP}]
	Since there exists a composition $\alpha$ such that $K_{\theta, (\lambda-\mu)|_{\alpha}} \neq 0$, and since $(1^d) \trianglelefteq \alpha$, Proposition~\ref{prop-refine-Kostka} ensures that $K_{\theta, (\lambda-\mu)|_{(1^d)}} \neq 0$.
	By Proposition~\ref{prop-E-expression}, we have that $E_{\lambda/\mu}^{\theta}$ is $p_{\lambda}$-positive, and therefore there must exist the term $p_{(1^d)} = s_{(d)}$ in its expansion. The proof now follows from Theorems~\ref{thm-Rado-Schur} and~\ref{thm-SNP-Schur-combin}.
\end{proof}

\section{Immanants of Giambelli matrices}\label{sec-Gm}
In this section, we present the proof of Theorem~\ref{thm-SNP-Giam}. Our approach relies heavily on the planar network construction introduced by Hamel and Goulden~\cite{HamelGoulden}.

Given a partition $\lambda = (\lambda_1, \lambda_2, \ldots, \lambda_{n})$ of rank $k$, as illustrated at the end of Section~\ref{sec-pre}, 
the outside decomposition is $\Pi = (\theta_1, \theta_2, \ldots, \theta_k)$, where 
\[
\theta_i = (\lambda_i - i + 1, 1^{\lambda'_i - i}),
\]
and the cutting strip of the Giambelli matrix $G_{\lambda}$ is the partition $\phi = (\lambda_1, 1^{n-1})$.

Assume that the contents in $\phi$ range from $c_1$ to $c_2$, and for $1 \le i \le k$, the contents of the starting box and the ending box of $\theta_i$ are $c_i'$ and $c_i''$, respectively.
Consider the directions of boxes in $\phi$, if a box with content $c$ (where $c_1 < c \leq c_2$) comes from the left (go right), then we set the direction of $x = c$ to be downward; conversely, if the box comes from the bottom (go up), then we set the direction of $x = c$ to be upward. Additionally, the directions of $x = c_1$ and $x = c_2 + 1$ can be chosen arbitrarily, as these choices will not affect the construction of our planar network.

For $1\le i\le k$, set $P_i = (c_i', -\infty)$ and $Q_i = (c_i'' + 1, -\infty)$ to be the starting points and ending points, respectively.
The weights of the planar network are assigned according to the following rules:
\begin{itemize}
	\item All vertical (up and down) steps have weight 1.
	\item The horizontal step from $(c_1, -j)$ to $(c_1 + 1, -j)$ has weight $x_j$.
	\item For $c_1 < c \le c_2$, suppose the horizontal step from $(c-1, j)$ to $(c, j)$ has weight $x_t$.
	\begin{itemize}
		\item If the direction at $x = c$ is upward, then the horizontal step from $(c, j)$ to $(c+1, j)$ has weight $x_{t - 1}$.
		\item If the direction at $x = c$ is downward, then the horizontal step from $(c, j)$ to $(c+1, j)$ has weight $x_t$.
	\end{itemize}
\end{itemize}

Following the aforementioned example, let $\lambda = (6,6,4,4,1,1)$, refer to Figure~\ref{fig-out-Gm},
the starting points and ending points are as follows:
\begin{align*}
	P_1 &= (-5, -\infty), & P_2 &= (-2, -\infty), & P_3 &= (-1, -\infty), & P_4 &= (0, -\infty), \\
	Q_1 &= (6, -\infty),  & Q_2 &= (5, -\infty),  & Q_3 &= (2, -\infty),  & Q_4 &= (1, -\infty).
\end{align*}

In Figure~\ref{fig-Gm-network}, we present all possible weights and those starting and ending points for the planar network associated with the Giambelli matrix $G_{\lambda}$, omitting the labels of steps with weight $1$.
Besides, we also present a family of lattice paths corresponding to the monomial ${\bf x}^{\alpha}=(x_1^6x_2x_3x_4x_5x_6)\cdot(x_1^5x_2x_3)\cdot(x_2^2x_3)\cdot x_3=x_1^{11}x_2^4x_3^4x_4x_5x_6$.
	
\begin{figure}[ht]
	\centering
			\begin{tikzpicture}
		[place/.style={thick,fill=black!100,circle,inner sep=0pt,minimum size=1mm,draw=black!100},scale=0.6]
		
		\foreach \y in {-5,-4,-3,-2,-1,0,1,2,3,4,5} {
			\foreach \x in {0,1,2,3,4,5,6,7} {
				\node (v\x\y) at (\x,\y) {};
			}
			
			\foreach \x in {-1,-2,-3,-4,-5,-6,-7} {
				\node (v\x\y) at (\x,\y) {};
			}
			
			\node (v8\y) at (8,\y) {};
			
			\draw (v-7\y) -- (v7\y);
		}
		
		\foreach \x in {0.5,1.5,2.5,3.5,4.5,5.5} {
			\foreach \y in {-4,-3,-2,-1,0} {
				\node at (\x,\y-0.25) {$x_{\the\numexpr(-\y)+1\relax}$};
			}
		}
		
		\foreach \x [count=\i from 0] in {-0.5,-1.5,-2.5,-3.5,-4.5,-5.5} {
			\foreach \y in {-4,-3,-2,-1,0} {
				\pgfmathsetmacro{\label}{int(-\y + \i + 1)}
				\node at (\x,\y-0.25) {$x_{\label}$};
			}
		}
		\foreach \x [count=\i from 0] in {-1.5,-2.5,-3.5,-4.5,-5.5} {
			\foreach \y in {1} {
				\pgfmathsetmacro{\label}{int( \y + \i )}
				\node at (\x,\y-0.25) {$x_{\label}$};
			}
		}
		\foreach \x [count=\i from 0] in {-2.5,-3.5,-4.5,-5.5} {
			\foreach \y in {2} {
				\pgfmathsetmacro{\label}{int( \y + \i-1 )}
				\node at (\x,\y-0.25) {$x_{\label}$};
			}
		}
		\foreach \x [count=\i from 0] in {-3.5,-4.5,-5.5} {
			\foreach \y in {3} {
				\pgfmathsetmacro{\label}{int( \y + \i-2 )}
				\node at (\x,\y-0.25) {$x_{\label}$};
			}
		}
		\foreach \x [count=\i from 0] in {-4.5,-5.5} {
			\foreach \y in {4} {
				\pgfmathsetmacro{\label}{int( \y + \i-3 )}
				\node at (\x,\y-0.25) {$x_{\label}$};
			}
		}
		\foreach \x [count=\i from 0] in {-5.5} {
			\foreach \y in {5} {
				\pgfmathsetmacro{\label}{int( \y + \i-4 )}
				\node at (\x,\y-0.25) {$x_{\label}$};
			}
		}
		\foreach \x in {-6,-5,-4,-3,-2,-1} {
			\draw [->] (\x,-5.5) -- (\x,5.5);
		}
		\foreach \x in {0,1,2,3,4,5,6} {
			\draw [->] (\x,5.5) -- (\x,-5.5);
		}
		\foreach \x in {0,1,2,3,4,5,6,-1,-2,-3,-4,-5,-6} {
			\pgfmathsetmacro{\label}{int(\x  + 1)}
			\node at (\x,-6) {\label};
		}
		\node at (-5-1,-7.2) {$P_1$};
		\node at (-2-1,-7.2) {$P_2$};
		\node at (-1-1,-7.2) {$P_3$};
		\node at (0-1,-7.2) {$P_4$};
		\node at (1-1,-7.2) {$Q_4$};
		\node at (2-1,-7.2) {$Q_3$};
		\node at (5-1,-7.2) {$Q_2$};
		\node at (6-1,-7.2) {$Q_1$};
		
		\draw [line width=1.2pt, purple]  (-6,-7) -- (-6,0);
		\draw [line width=1.2pt, purple]  (-3,-7) -- (-3,0);
		\draw [line width=1.2pt, orange]  (-2,-7) -- (-2,-1);
		\draw [line width=1.2pt, orange]  (1,-7) -- (1,-1);
		\draw [line width=1.2pt, purple]  (4,-7) -- (4,0);
		\draw [line width=1.2pt, purple]  (5,-7) -- (5,0);
		\draw [line width=1.2pt, purple] (-6,0) -- (5,0);
		\draw [line width=1.2pt, teal]  (-1,-7) -- (-1,-2);
		\draw [line width=1.2pt, teal]  (0,-7) -- (0,-2);
		\draw [line width=1.2pt, teal] (-1,-2) -- (0,-2);
		\draw [line width=1.2pt, orange] (-2,-1) -- (1,-1);
	\end{tikzpicture}
	\caption{A planar network of $G_{\lambda}$.}
	\label{fig-Gm-network}
\end{figure}
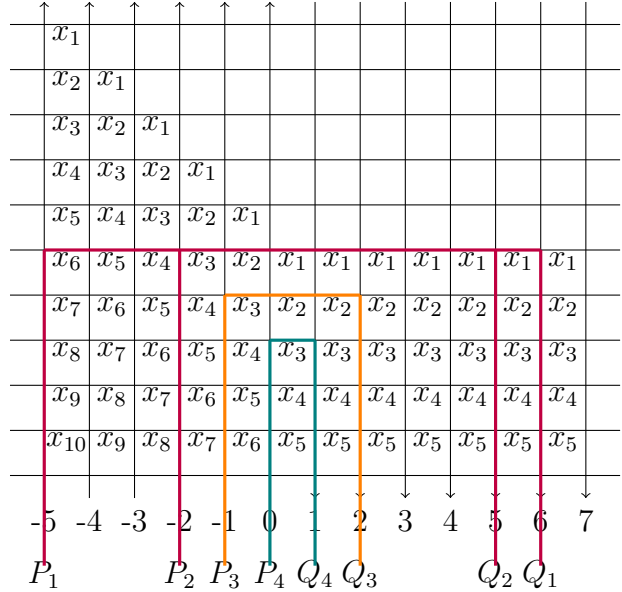

Now we can give the proof of Theorem~\ref{thm-SNP-Giam}.

\begin{proof}[Proof of Theorem~\ref{thm-SNP-Giam}]
Since the cutting strip of the Giambelli matrix $G_{\lambda} $ is the largest hook of $\lambda$, the boxes from the starting box to the ending box must be first go up then go right. According to the rules of the planar network construction, the directions of all vertical steps must be first all upward then all downward from the left to the right.
Specifically, let $c_1$ and $c_2$ be the contents of the starting box and the ending box, respectively, then the direction of each $x=c$ is upward if $c_1\le c\le0$, and is downward if $0< c \le c_2$.
Thus, on the one hand, when considered horizontally, the weights of the horizontal steps to the left of the line $x=0$ decrease step by step from right to left, while those to the right of $x=0$ remain constant on each level. On the other hand, when viewed vertically, the weights of the horizontal steps form a sequence that decreases step by step from top to bottom.

Suppose that $\nu = (\nu_1, \nu_2, \ldots, \nu_t)$.
By the construction of the planar network for $G_{\lambda}$ described before, we observe that the lattice paths corresponding to the largest term (in dominance order) in the monomial expansion of $\mathrm{Imm}_{\nu}G_{\lambda}$ are distributed as high as possible and have type $\nu$. 
That is, the lattice paths $\{P_1 \rightarrow Q_1, \ldots, P_{\nu_1} \rightarrow Q_{\nu_1} \}$ overlap on the row containing the most steps weighted by $x_1$, and for $2 \le i \le t$, the lattice paths 
\[
\left\{P_{\sum_{j=1}^{i-1}\nu_j + 1} \rightarrow Q_{\sum_{j=1}^{i-1}\nu_j + 1}, \ldots, P_{\sum_{j=1}^{i}\nu_j} \rightarrow Q_{\sum_{j=1}^{i}\nu_j} \right\}
\] 
overlap on the row containing the most steps weighted by $x_i$, respectively.


Similar to the proof of Theorem~\ref{thm-SNP-JT}, by applying the expression in \eqref{eq-imm-expression} and Proposition~\ref{prop-skleton}, it suffices to prove that 
\[
\chi^{\nu}\left(\mathfrak{S}_{\{1,2,\ldots,\nu_1\}} \mathfrak{S}_{\{\nu_1+1,\ldots,\nu_1+\nu_2\}} \cdots \mathfrak{S}_{\{\sum_{j=1}^{t-1}\nu_j+1,\ldots,k\}} \right) \neq 0,
\]
and $\chi^{\nu}(\mathfrak{S}_{\mu}) = 0$ for any $\mu \trianglerighteq \nu$ in the dominance order. 
This now follows from Lemma~\ref{lem-Young_char}, which completes the proof.
\end{proof}

 \section{Further work}\label{sec-further-work}

In this paper, for immanants of Jacobi--Trudi matrices, we have only proved their SNP property for several special partitions (namely $\nu=(n)$ and $\nu=(n-1,1)$ under a border strip condition). It is natural to ask whether this property holds in general. Based on computer experiments conducted in Sage~\cite{sage}, we propose the following conjecture.

\begin{conj}\label{conj-1}
	For a skew partition $\lambda/\mu$ of length $n$ and any partition $\nu \vdash n$, each immanant $\operatorname{Imm}_\nu H(\lambda,\mu)$ (in finitely many variables) is SNP.
\end{conj}

For now, we have established the SNP property for several families of immanant-derived polynomials and shown that their Newton polytopes are special classes of generalized permutohedra. In fact, these polynomials satisfy a stronger property than SNP, known as M-convexity.

A set $J \subset \mathbb{N}^n$ is called M-convex if for any $\alpha,\beta \in J$ and any index $i$ with $\alpha_i > \beta_i$, there exists $j$ such that $\alpha_j < \beta_j$ and $\alpha - e_i + e_j \in J$. A polynomial is M-convex if its support is M-convex. For a homogeneous polynomial, M-convexity is equivalent to the SNP property together with the condition that its Newton polytope is a generalized permutahedron~\cite[Remark 4.1.1]{StD2020}. For further background, see~\cite{POS2009, Mur03}. For more background, see~\cite{POS2009, Mur03}.
Moreover, there is a close connection between M-convex polynomials and Lorentzian polynomials, the latter of which was introduced by Br\"and\'en and Huh \cite{BH2022} as a generalization of volume polynomials and stable polynomials. 
Specifically, we say that a homogeneous polynomial $f(x_1,\ldots,x_n)$ of degree $d$ with nonnegative coefficients is a Lorentzian polynomial if and only if it satisfies:
 \begin{itemize}
 	\item The support of $f$ forms an M-convex set.
 	\item The quadratic form $\partial_{i_1} \cdots \partial_{i_{d-2}}f $ has at most one positive eigenvalue for any index $i_1,\ldots, i_{d-2}\in[n]$.
 \end{itemize}
 Br\"and\'en and Huh proved that if ${\rm supp}(f)$ is M-convex then the polynomial 
 \[\sum_{\alpha\in{\rm supp}(f)}\frac{x_1^{\alpha_1}}{\alpha_1!}\cdots\frac{x_{n}^{\alpha_{n}}}{\alpha_{n}!}\]
 is Lorentzian; see \cite[Theorem 3.10]{BH2022} for more information.
Inspired by the above result, they defined the normalization of a polynomial $f=\sum_{\alpha=(\alpha_1,\ldots,\alpha_n)}c_{\alpha}x^{\alpha}$ as
 \[N(f)=\sum_{\alpha\in{\rm supp}(f)}c_{\alpha}\frac{x_1^{\alpha_1}}{\alpha_1!}\cdots\frac{x_{n}^{\alpha_{n}}}{\alpha_{n}!}.\]
Later, Huh, Matherne, M\'{e}sz\'{a}ros, and St.~Dizier~\cite{HMMD2022} conjectured that $N(s_{\lambda/\mu})$ is Lorentzian for any skew partition $\lambda/\mu$ and proved this conjecture when $\mu=\emptyset$. 
Motivated by their conjecture, and Theorem~\ref{thm-SNP-JT}, we consider the following problem.
 
 \begin{prob}
 	Is it true that the normalization of each immanant of Jacobi–Trudi matrix is a Lorentzian polynomial?
 \end{prob}
 
Unfortunately, the answer is negative in general. For example, take the skew partition $\lambda/\mu=(3,2)/(1)$ and the partition $\nu=(2)$, one can compute that
 \[{\rm Imm}_{\nu}H(\lambda,\mu)=7m_{(1,1,1,1)}+5m_{(2,1,1)}+4m_{(2,2)}+3m_{(3,1)}+2m_{(4)},\]
 and the quadratic form $\partial_i^2 N({\rm Imm}_{\nu}H(\lambda,\mu)(x_1,x_2,x_3,x_4))$ has two positive eigenvalue for $i\in[4]$.
 
 Nevertheless, for the polynomials $E_{\lambda/\mu}^{\theta}$ associated to the Jacobi–Trudi matrix $H(\lambda,\mu)$, we conjecture that their normalizations are Lorentzian. We have verified this conjecture for all border strips with at most 10 boxes.
 
  \begin{conj}\label{conj-2}
 	Given a border strip $\lambda/\mu$ and a partition $\theta$ of the same size. The polynomial $N(E_{\lambda/\mu}^{\theta})$ is Lorentzian.
 \end{conj}
 
 For immanants of Giambelli matrices, based on Theorem~\ref{thm-SNP-Giam}, we also propose the following conjecture and verify it for all possible partitions of size less than or equal to $9$ by using SageMath \cite{sage}. 

 \begin{conj}\label{conj-3}
	 Given a positive integer $k$, let $\lambda$ be a partition of rank $k$. Then for any partition $\nu$ of $k$, the polynomial $N({\rm Imm}_{\nu}G_{\lambda})$ is Lorentzian.
	 \end{conj}

	\vskip 0.5cm
\noindent \textbf{Acknowledgments.} This work is partially supported by
the Strategic Priority Research Program of the Chinese Academy of Sciences (No. XDB0510201), the NSFC grant (No.\ 12271511), and the Postdoctoral Fellowship Program of CPSF under Grant Number: GZC20241868. The author would like to thank Shaoshi Chen, Ethan Y.H. Li and Xin-Bei Liu for the helpful discussions.

\bibliography{sn-bibliography}

\end{document}